\documentclass[12pt]{amsart}
\usepackage{amsmath}
\usepackage{amssymb}
\usepackage{eucal}
\usepackage{amsthm}
\usepackage{amscd}
\usepackage{hyperref}
\usepackage{fullpage}
\usepackage{url}

\usepackage[]{amsmath, amsthm, amsfonts,graphicx, amscd,}
\usepackage[all,cmtip]{xy}

\begin{document}

\newtheorem {thm}{Theorem}[section]
\newtheorem{corr}[thm]{Corollary}
\newtheorem {cl}[thm]{Claim}
\newtheorem*{thmstar}{Theorem}
\newtheorem{prop}[thm]{Proposition}
\newtheorem*{propstar}{Proposition}
\newtheorem {lem}[thm]{Lemma}
\newtheorem*{lemstar}{Lemma}
\newtheorem{conj}[thm]{Conjecture}
\newtheorem{question}[thm]{Question}
\newtheorem*{questar}{Question}
\newtheorem{ques}[thm]{Question}
\newtheorem*{conjstar}{Conjecture}

\theoremstyle{remark}
\newtheorem{rem}[thm]{Remark}
\newtheorem{np*}{Non-Proof}
\newtheorem*{remstar}{Remark}

\theoremstyle{definition}
\newtheorem{defn}[thm]{Definition}
\newtheorem*{defnstar}{Definition}
\newtheorem{exam}[thm]{Example}
\newtheorem*{examstar}{Example}
\newtheorem{convention}[thm]{Convention}

\newenvironment{pfcl}{\vspace{.1in}  \noindent {\bf Proof of Claim:}}{\hspace{\fill} $\maltese$ \vspace{.1in}}

\newcommand{\PSL}{\operatorname{PSL}}
\newcommand{\pd}[2]{\frac{\partial #1}{\partial #2}}
\newcommand{\pp}{\partial }
\newcommand{\pdtwo}[2]{\frac{\partial^2 #1}{\partial #2^2}}
\newcommand{\od}[2]{\frac{d #1}{d #2}}
\def\Ind{\setbox0=\hbox{$x$}\kern\wd0\hbox to 0pt{\hss$\mid$\hss} \lower.9\ht0\hbox to 0pt{\hss$\smile$\hss}\kern\wd0}
\def\Notind{\setbox0=\hbox{$x$}\kern\wd0\hbox to 0pt{\mathchardef \nn=12854\hss$\nn$\kern1.4\wd0\hss}\hbox to 0pt{\hss$\mid$\hss}\lower.9\ht0 \hbox to 0pt{\hss$\smile$\hss}\kern\wd0}
\def\ind{\mathop{\mathpalette\Ind{}}}
\def\nind{\mathop{\mathpalette\Notind{}}}
\newcommand{\m}{\mathbb }
\newcommand{\mc}{\mathcal }
\newcommand{\mf}{\mathfrak }
\newcommand{\is}{^{p^ {-\infty}}}
\newcommand{\QQ}{\mathbb Q}
\newcommand{\fh}{\mathfrak h}
\newcommand{\CC}{\mathbb C}
\newcommand{\RR}{\mathbb R}
\newcommand{\ZZ}{\mathbb Z}
\newcommand{\tp}{\operatorname{tp}}
\newcommand{\SL}{\operatorname{SL}}
\newcommand{\UU}{\mathbb U}
\renewcommand{\AA}{\mathbb A}
\newcommand{\GL}{\operatorname{GL}}
\newcommand{\cM}{{\mathcal M}}
\newcommand{\cR}{{\mathcal R}}

\title{Strong minimality and the $j$-function}
\author{James Freitag}
\email{freitag@math.berkeley.edu}

\thanks{JF is partially supported by an NSF MSPRF. TS is partially supported by NSF Grant DMS-1001550 and DMS-1363372.
This material is based upon work supported by the National Science Foundation under Grant No. 0932078 000 while the authors
were in residence at the Mathematical Sciences Research Institute in Berkeley, California, during the Spring 2014 semester.
JF thanks Ronnie Nagloo for discussions about applications of Seidenberg's theorem.  TS thanks
Jonathan Pila for discussions about his Ax-Lindemann-Weierstrass theorem. Both authors thank Barry Mazur for posing the questions which we answer in the last section and for his numerous discussions with us about this work.  The authors thank Anand Pillay for suggesting some improvements and for
discussing issues around the effective finiteness theorem for differential fields}
\author{Thomas Scanlon}
\email{scanlon@math.berkeley.edu}
\address{University of California, Berkeley \\
Department of Mathematics \\
Evans Hall \\
Berkeley, CA 94720-3840 \\
USA}
\maketitle

\begin{abstract}
We show that the order three algebraic differential equation over $\QQ$ satisfied by the
analytic $j$-function defines a non-$\aleph_0$-categorical strongly minimal set with trivial
forking geometry relative to the theory of differentially closed fields of characteristic zero
answering a long-standing open problem about the existence of such sets.   The theorem follows
from Pila's modular Ax-Lindemann-Weierstrass with derivatives theorem using Seidenberg's embedding
theorem and a theorem of Nishioka on the differential equations satisfied by automorphic functions. As a by product of this analysis, we obtain a more general version of the modular Ax-Lindemann-Weierstrass theorem, which, in particular, applies to automorphic functions for arbitrary arithmetic subgroups of $SL_2 (\m Z)$. We then apply the results to prove effective finiteness results for intersections of subvarieties of products of modular curves with isogeny classes. For example, we show that if $\psi:{\mathbb P}^1 \to {\mathbb P}^1$ is any non-identity automorphism of the
projective line and $t \in {\mathbb A}^1(\CC) \smallsetminus {\mathbb A}^1({\mathbb Q}^\text{alg})$, then
the set of $s \in {\mathbb A}^1(\CC)$ for which the elliptic curve with $j$-invariant $s$ is isogenous to the elliptic curve
with $j$-invariant $t$ and the elliptic curve with $j$-invariant $\psi(s)$ is isogenous to the elliptic curve with $j$-invariant
$\psi(t)$ has size at most $36^7$.  In general, we prove that if $V$ is a Kolchin-closed subset of $\m A^n$, then the Zariski closure of the intersection of $V$ with the isogeny class of a tuple of transcendental elements is a finite union of weakly special subvarieties. We bound the sum of the degrees of the irreducible components of this union by a function of the degree and order of $V$. 
\end{abstract}

\section{Introduction}

According to Sacks,  ``[t]he least misleading example of a totally transcendental theory is the theory of differentially closed
fields of characteristic 0 ($\operatorname{DCF}_0$)''~\cite{Sacks}.  This observation has been borne out through the discoveries
that a prime differential field need not be minimal~\cite{notmin}, the theory $\operatorname{DCF}_0$ has the ENI-DOP
property~\cite{DOP}, and Morley rank and Lascar rank differ in differentially closed fields~\cite{MneqU}, amongst others.   However,
the theory of differentially closed fields of characteristic zero does enjoy some properties not shared by all totally transcendental
theories, most notably the Zilber trichotomy holds for its minimal types~\cite{HrSo} and there are infinite definable families of
strongly minimal sets for which the induced structure on each such definable set is $\aleph_0$-categorical and orthogonality between
the fibers is definable~\cite{HrushovskiJou,HrIt}.   Early in the study of the model theory of differential fields, Lascar asked whether
it might be the case that the induced structure on every strongly minimal set orthogonal to the constants must be
$\aleph_0$-categorical~\cite{Lascar} (Lascar's formulation of the question was slightly different, though implies the condition we stated. Also, Lascar attributed the question to Poizat, but the question does not appear in the paper to which Lascar refers.).  From the existence of Manin kernels, one knows that there are strongly minimal sets relative to
$\operatorname{DCF}_0$ which are not $\aleph_0$-categorical~\cite{HrML}, but the question of whether there are non-$\aleph_0$-categorical
strongly minimal sets with trivial forking geometry has remained open~(see \cite{RosenJou} for instance).  We exhibit an explicit equation defining
a set with such properties.

The analytic $j$-function, $j:\fh \to \CC$, which has been known to mathematicians for quite some time, appearing implicitly in the
work of Gauss already in the late Eighteenth Century~\cite{Gaussisold}, satisfies a differential equation over
$\QQ$  which when evaluated in a differentially closed field defines a non-$\aleph_0$-categorical strongly minimal set
with trivial forking geometry. The specific differential equation satisfied by the $j$-function is given by the vanishing of a differential rational function. In addition to the fiber of the function above zero (the equation of the $j$-function), we prove that the other fibers are also strongly minimal, trivial, and pairwise orthogonal. 

Besides the applications to differential algebraic geometry, we give some number theoretic applications. Specifically, we use our differential algebraic approach to prove effective bounds on the size of the intersection of Hecke orbits of transcendental points on products of modular curves with non-weakly special varieties.

Mazur posed some effective finiteness questions in connection with a recent theorem of Orr~\cite{orr2012families}
and his program to find very compact invariants for elliptic curves.  Of course, knowing the isogeny class of an elliptic curve determines
that curve only to within a countably infinite set.  Mazur surmised that the data of the isogeny class of an elliptic curve $E$ and of the
isogeny class of some other naturally associated (but not so naturally associated as to respect the Hecke correspondences) elliptic curve
might pin down $E$ or at least constrain it to a finite set.   In fact, it is a consequence of the main theorem of~\cite{orr2012families} that
if $C \subseteq {\mathbb A}^2_\CC$ is an irreducible affine plane curve which is not modular or horizontal or vertical,
then for any point $(a,b) \in C(\CC)$ there
are only finitely many other points $(c,d) \in C(\CC)$ for which the elliptic curve coded by $a$ is isogenous to the elliptic curve given by $c$ and
the curve corresponding to $b$ is isogenous to that coded by $d$.   In this sense, if we regard $C$ as a correspondence  which 
associates to an elliptic curve $E$ with $j$-invariant $a$ one of the elliptic curves having $j$-invariant $b$ with $(a,b) \in C(\CC)$, then
the data of the isogeny class of $E$ and of the $C$-associated elliptic curve determined $E$ up to a finite set.

Orr's theorem applies to arbitrary points without any hypothesis on the degree of the point over $\QQ$.  However,
this generality incurs a cost in that his argument follows the Pila-Zannier strategy for proving diophantine geometric
finiteness theorems which depends in an essential way on ineffective results in the Pila-Wilkie o-minimal
counting theorem and in class field theory.
On the other hand, by restricting attention to transcendental points, we may compute explicit bounds on the sizes of these finite sets.
While our proof that the sets in question are finite also passes through the Pila-Wilkie o-minimal counting theorem in the guise
of Pila's modular Ax-Lindemann-Weierstrass theorem with derivatives, this appeal does not leave a trace of ineffectivity.

Let us describe the basic tactics involved in our approach to the general problem.
The key point is to replace the Hecke orbits by solutions sets to particular differential equations.
This approach already appears in Buium's article~\cite{buiumIII}.
 The obvious downside to this move is that (as referenced above) the Kolchin (differential Zariski) topology has wild behavior
 compared to the Zariski topology. This is mitigated by our model theoretic work understanding the differential equation satisfied
 by the $j$-function. Here we use the strong minimality and triviality of the differential equation which the $j$-function satisfies
 in order to establish the finiteness of its intersection with non-weakly special algebraic varieties. The advantage of this
 approach is uniformity - we replace an arithmetic object by a (differential) variety. The finiteness of certain intersections
 then follows by our proof of strong minimality, and the actual bounds come from an effective version of uniform bounding
 for definable sets in differential fields due to Hrushovki and Pillay \cite{HPnfcp}; essentially these bounds come from
 doing intersection theory (of algebraic varieties) in jet spaces of algebraic varieties. The actual bounds are rather
  tractable, being doubly exponential in the various inputs - the degrees and dimensions of certain associated algebraic varieties.

This paper is organized as follows.  In Section~\ref{basicthy}, we recall some of the basic theory of the $j$-function, including
the theory of the Schwarzian derivative and the differential equation satisfied by $j$.   In Section~\ref{Nishi} we recount a theorem
of Nishioka showing that the $j$-function does not satisfy any nonzero algebraic differential equations over $\CC$ of order
two or less.  With Section~\ref{minsect} we complete the proof of our main theorem and draw some corollaries.   The main ingredients
of the proof, in addition to Nishioka's theorem, are Seidenberg's embedding theorem, Pila's modular Ax-Lindemann-Weierstrass with
derivatives theorem and a construction a nonlinear order three differential rational operator $\chi$ for which 
$\chi(j) = 0$.  In Section~\ref{otherfibers}, we show via a change of variables trick and some basic forking calculus that 
for any parameter $a$ the set defined by $\chi(x) = a$ is strongly minimal.  In Section~\ref{orthosect} we show that these
fibers are orthogonal.  
The final section is devoted to arithmetic applications, where we use our main theorem to get bounds on the intersections of non-weakly special varieties with Hecke orbits in products of modular curves.

\section{Basic theory of the $j$-function} \label{basicthy}

In this section we summarize some of the basic theory of the $j$-function.

We denote the upper half-plane by
$$
\fh := \{ z \in \CC ~:~ \operatorname{Im}(z) > 0 \} \text{ .}
$$

We write $t$ for the variable ranging over $\fh$ (or some open subdomain).

The $j$-function is an analytic function on $\fh$ whose Fourier
expansion begins with
$$
j(t) = \exp(-2 \pi i t) + 744 + 196~884 \exp(2 \pi i t) + 21~493~760 \exp(4 \pi i t) + \cdots
$$

The algebraic group $\SL_2(\CC)$ acts on the projective line via linear fractional transformations and the restriction
of this action to $\SL_2(\RR)$ induces an action of $\SL_2(\RR)$ on $\fh$.  The $j$-function is
a modular function for $\operatorname{SL}_2(\ZZ)$ in the sense that $j(\gamma \cdot t) = j(t)$ for
each $\gamma \in \operatorname{SL}_2(\ZZ)$.  Indeed, more is true:  for $a, b \in \fh$ one has $j(a) = j(b)$ if and only if there
is some $\gamma \in \SL_2(\ZZ)$ with $\gamma \cdot a = b$.

The differential equation satisfied by $j$ is best expressed using the Schwarzian derivative.   We shall write $x'$ for $\pd{x}{t}$.
More generally, in any differential ring $(R,\partial)$ we shall write $x'$ for $\partial(x)$.
We define the Schwarzian by

$$
S(x)=\left( \frac{x''}{x'}\right) '-\frac{1}{2} \left(\frac{x''}{x'}\right)^2 \text{ .}
$$

When dealing with the Schwarzian derivative associated with a particular derivation $\partial$, we will use the notation $S_\partial$, but when $\partial$ is fixed or clear from the context, we will drop the subscript. 

The Schwarzian satisfies a chain rule:

$$
S(f \circ g) = (g')^2 S(f) \circ g  + S(g) \text{ .}
$$

A characteristic feature of the Schwarzian is that if $(K,\partial)$ is a differential field of characteristic zero with
field of constants $C = \{ x \in K ~:~ x' = 0 \}$ and $f$ and $g$ are two elements of $K$, then 
one has $S(f) = S(g)$ if and only if $f = \frac{ag + b}{cg + d}$ for some constants $a$, $b$, $c$ and $d$.  In particular,
one computes immediately from the formula for the Schwarzian that if $z' = 1$, then $S(z) = 0$ so that 
the solutions to the equation $S(x) = 0$ are precisely the degree one rational functions in $z$ with 
coefficients from $C$.

The following is an order three algebraic differential equation satisfied by $j$ (see~\cite[page 20]{masserheights}).

$$E:= S(y) + R(y) (y')^2 = 0 \text{ ,}$$

where $$R(y) = \frac{y^2 - 1968 y + 2~654~208}{2 y^2 (y - 1728)^2} \text{ .}$$

For the remainder of this paper, when we speak of the differential equation satisfied by $j$, we mean $E$.
We will also make use of the differential rational function which gives the equation; throughout the paper, we will denote
$$\chi (y) : = S(y) + R(y) (y')^2 \text{ .}$$

Similarly, when there is some ambiguity or choice about the particular derivation 
with which we are working, we will write $\chi_\delta$ for the resulting differential rational function.

\section{Nishioka's method and automorphic functions} \label{Nishi}
Nishioka~\cite{Nishioka} proved a conjecture of Mahler~\cite{mahler1969algebraic} regarding automorphic
functions and the differential equations that they might satisfy. In this section, we explain Nishioka's
result and review its proof here, noting certain necessary uniformities.

Let us recall the notion of an automorphic function, specifically adapted to Nishioka's method.  In our application, we take
$G = \operatorname{SL}_2({\mathbb Z})$, $D = \fh$  and $f = j:\fh \to \CC$ the
analytic $j$-function.  In this section we shall write $t$ for the variable ranging over $D$ and $T$ for the variable ranging
over $G$.

\begin{defn} Let $G \leq \operatorname{SL}_2 (\m C)$ be a subgroup.  A function $f(t)$ which is analytic on some domain $D \subset \m P^1$ is called automorphic if it satisfies the following properties:
\begin{enumerate}
\item For all $T \in G$ and all $t \in D$ one has $Tt \in D$.
\item $f(Tt) = f(t)$ for $t \in D$.
\end{enumerate}
\end{defn}

The main theorem of~\cite{Nishioka} is the following.

\begin{thm}  \label{NishiokaSpecial} Let $G$ be a Zariski dense subgroup of $\operatorname{SL}_2 ( \m C)$.
Then any nontrivial automorphic function of $G$ satisfies no algebraic differential equation of order two or less over $\m C $.
\end{thm}

\begin{rem}
In~\cite{Nishioka} the hypothesis in Theorem~\ref{NishiokaSpecial} is that $G$ have at least three limit points rather than that it be Zariski dense
in $\operatorname{SL}_2({\mathbb C})$, but  these conditions are equivalent as is noted in~\cite{Nishioka}.
\end{rem}

\begin{rem}
The conclusion of Theorem~\ref{NishiokaSpecial} is ostensibly stronger in that $f$ 
does not satisfy low order differential equations even
over ${\mathbb C}(t,\exp(2 \pi i t))$.   The inclusion of these additional parameters is a red herring. 
We explain in Remark~\ref{Nishiokagen} 
how such independence follows from the minimality of $\operatorname{tp}(j/{\mathbb C})$.
\end{rem}

There are some minor errors in Nishioka's proof of Theorem~\ref{NishiokaSpecial} \cite[page 46]{Nishioka}.  The first is very slight:
there is a misplaced reference to Lemma~4 of his paper (which should be to Lemma~3).  Somewhat more seriously,
the necessary uniformity of the algebraic dependence in the first half of his proof is not mentioned at all.  For the sake
of completeness, we reproduce his proof with these defects remedied.

\begin{proof} Let $f(t)$ be automorphic for $G$. Assume that $t, f(t) , \od{}{t} f(t) , \od{^2}{t^2} f(t)$ are algebraically dependent over $ \m C$; specifically, there is some nonzero polynomial $F$ with constant coefficients for which $$F(t, f(t) , \od{}{t} f(t) , \od{^2}{t^2} f(t))=0 \text{ .}$$ Then for each $T \in G$, the same is true when we substitute the variable $Tt$ for $t$. Now our four functions have become:
\begin{enumerate}
\item $Tt$
\item $f(Tt) = f(t)$
\item $\od{}{Tt} f(Tt) = (\od{}{t} Tt ) ^{-1} \od {}{t} f(t)$
\item $(\od{}{Tt})^2 f(Tt) = (\od {}{t} Tt )^{-2} \od {^2}{t^2}f(t) - (\od{}{t} Tt)^{-3} (\od{^2}{t^2}Tt) \od{}{t} f(t)$.
\end{enumerate}
So, by our earlier remarks, we have that
$$F(Tt, f(t) , (\od{}{t} Tt ) ^{-1} \od {}{t} f(t) , (\od {}{t} Tt )^{-2} \od {^2}{t^2}f(t) - (\od{}{t} Tt)^{-3} (\od{^2}{t^2}Tt) \od{}{t} f(t))=0 \text{ .}$$
By clearing denominators, we obtain a nonzero polynomial over  ${\mathbb C}(t)$ which vanishes on the
triple $(Tt, \od{}{t} Tt, \od{^2}{t^2} Tt)$ for all $T \in G$.  This violates Lemma~3 of~\cite{Nishioka}.
\end{proof}

\begin{rem}\label{Nishiokagen} Nishioka actually proves a slightly stronger conclusion; namely, that the automorphic function satisfies no differential equations of order two over $\m C ( t, e^t)$. In this paper, we will be essentially concerned with the $j$-function (and automorphic functions for other arithmetic subgroups), where this part of the theorem will be an easy consequence of strong minimality.
\end{rem}

\section{Minimality and the $j$-function} \label{minsect}

In this section, we deduce our main theorem on the strong minimality of the
set defined by $E$.   Model theoretic notation is standard and generally follows that
of~\cite{GST}.   We regard the differential field $\CC \langle j \rangle = \CC(j,j',j'')$
as a subdifferential field of some differentially closed field with field of constants $\CC$.

Let us recall Seidenberg's embedding theorem~\cite{seidenberg}.

\begin{thm} \label{Seidenberg} Let $K = \m Q \langle u_1 , \ldots , u_n \rangle$ be a differential field
generated by $n$ elements over $\QQ$ and let $K_1 = K \langle v \rangle$ be
a simple differential field extension of $K$.  Suppose $U \subset \m C$ is an open ball and $\iota : K \rightarrow \cM(U)$
is a differential field embedding of $K$ into the differential field of meromorphic functions on $U$.
Then there is an open ball $ V \subseteq U$ and an extension of $ \iota $ to a differential field embedding of $K_1$ into $\cM(V)$.
\end{thm}

Let us recall a basic principle in stability theory, called the ``Shelah reflection principle'' in~\cite{ChHr-AD1}.
In a stable theory, if $A \subseteq B$ is an extension of subsets of some model and $a$ is any tuple,
then if $\tp(a/B)$ forks over $A$, one can find a canonical base for $\tp(a/B)$ within the algebraic closure
of an initial segment of a Morley sequence in $\tp(a/B)$, $\{ d_i \} _{i \in I}$.
In a superstable theory, the initial segment is finite. Specifically, this implies that the (still indiscernible over $A$) sequence $\{ d_i \}_{i = 1}^n$
is not independent over $A$. A proof of this principle in the more general context of simple theories may be found in \cite[Proposition 17.24]{Casanovas}. A proof in the stable case may be found in \cite[Lemma 2.28]{GST}.

\begin{lem} \label{RLEM} We can realize any finite indiscernible sequence in  $\tp ( j (t )) / \m C)$ via $\{j(g_i t ) \}$ where $g_i \in GL_2 (\m C)$.
\end{lem}

\begin{rem}
The first author discussed portions of this argument with Ronnie Nagloo, who made several essential suggestions.
\end{rem}

\begin{proof}
 By using Theorem~\ref{Seidenberg}, we may assume that the initial segment of the Morley sequence,
 from which we extract the canonical base $\{d_1, \ldots , d_n \}$ is embedded in the field of meromorphic
 functions on some open domain $U$ contained in $\fh$. Since the $j$-function is a surjective
 analytic function from $\fh \rightarrow \m C$, it follows that
 there are holomorphic functions $\psi_i: U \rightarrow \fh$ such that $j( \psi_i (t)) = d_i (t)$.
 Of course, we know $j( \psi_i (t))$ satisfies the same differential equation as $j(t)$, namely,
\begin{eqnarray*}
0 & = &  \chi (j \circ  \psi_i ) \\
   &=& S(j \circ \psi_i) + R( j \circ \psi_i ) ( (j \circ \psi_i)')^2  \\
  & =& (S(j) \circ   \psi_i ) \cdot (\psi_i')^2 + S(\psi_i)+R( j \circ \psi_i) (j' \circ \psi_i)^2 \cdot ( \psi_i') ^2  \\
& = & (\chi (j) \circ \psi_i ) \cdot (\psi_i' )^2 + S (\psi_i)  \\
& = & S(\psi_i) \text{ .}
\end{eqnarray*}

So, we can see that if $j \circ \psi _i$ is a solution to $\chi (x) =0$, then
$S(\psi_i)=0$.  As we noted above, all such solutions are rational functions of degree one.
That is, $j(\psi_i (t))=j(g_i t)$ where $g_i \in \operatorname{GL}_2 (\m C)$.

\end{proof}

With the next theorem we deduce from Pila's module Ax-Lindemann-Weierstrass with derivatives theorem that $\tp(j/\CC)$ is minimal.

\begin{thm}\label{Urank1} $\operatorname{RU} (\tp (j /\m C))=1$.
\end{thm}
\begin{proof} By the stability of $\operatorname{DCF}_0$, we 
may find a countable algebraically closed subfield $A \subseteq {\mathbb C}$  
for which $\operatorname{RU}(j/A) = \operatorname{RU}(j/{\mathbb C})$. 

We need to check that any forking extension of $\tp( j / A)$ is algebraic.  By the finite character of forking, 
it suffices to consider extensions of the type to finitely generated extensions of $A$. 
If $B \supseteq A$ is any such finitely generated differential field
extension in our differentially closed field for which $\tp(j/B)$ forks over
$A$, then by the Shelah reflection principle mentioned above, we may find a finite Morley sequence
$\{ d_1, \ldots, d_n \}$ in $\tp(j/B)$ which is not independent over $A$.  By Lemma~\ref{RLEM}, we may realize $d_1, \ldots, d_n$
as $j(g_1 t), \ldots, j(g_n t)$ for some $g_i \in \operatorname{GL}_2(\CC)$.  The modular Ax-Lindemann-Weierstrass with derivatives
theorem  of~\cite{PilaDer} asserts that if $g_1, \ldots, g_n$ are in distinct cosets of $\operatorname{GL}_2(\QQ)$, then
$j(g_1 t), \ldots, j(g_n t)$ are independent over $\CC$.  However, if $g_i$ and $g_j$ are in the same coset
of $\operatorname{GL}_2(\QQ)$, then $j(g_i t)$ and $j(g_j t)$ are interalgebraic over $\QQ$ as witnessed
by an appropriate modular polynomial $F_N (x,y)$ \cite[see pages 183-186]{MilneEC}; we will refer to this relation (between solutions of the differential equation $E$) as a Hecke correspondence; Pila \cite{PilaDer, PilaAO} calls these modular relations. The only way that the elements of a Morley sequence may be interalgebraic is if the type itself is
algebraic.  Hence, from the dependence of the Morley sequence we deduce that $\tp(j/B)$ is algebraic, as required.
\end{proof}

Using Nishioka's theorem, we strengthen Theorem~\ref{Urank1} to the conclusion that $E$ defines a \emph{strongly} minimal set.

\begin{thm}\label{stronglyminimal} The set $X$ defined by the differential equation $E$ is strongly minimal.
\end{thm}
\begin{proof} As the equation $E$ has degree one in order three, it
suffices to show that any differential specialization of $j$ over $ \m C$ satisfies no lower order differential equation.
By the proof of Lemma~\ref{RLEM} and Theorem~\ref{Seidenberg}, and the fact that any differential specialization $f$ satisfies the equation
$$S(f) + R(f) ( f')^2 =0$$
one can assume that $f=j(gt)$ for some $g \in \operatorname{GL}_2 ( \m C)$. Now $f$ satisfies the hypotheses of
Theorem~\ref{NishiokaSpecial} (with $G=\SL_2(\m Z)^g = g^{-1} \SL_2(\m Z) g$)
 and so it satisfies no nontrivial order two or less equation over $\m C$.
\end{proof}

\begin{rem} Recall that we remarked in Remark~\ref{Nishiokagen} that Nishioka proved a slightly more general statement than
that an automorphic form $f(t)$ satisfies no order two differential equation over $ \m C$.
In fact, he proves that the conclusion holds over $\m C (t, e^t)$. For the $j$-function, this conclusion follows from minimality of the type
$\tp(j/\CC)$. Indeed, this depends very little on the functions $t$ and $e^t$. The same conclusion holds for any function (or collection of functions) $f(t)$ so that $f(t)$ satisfies an order two (or lower) differential equation over $\m C$.  To see this, we show
by induction on $n$ that if $f_1, \ldots, f_n$ is a finite sequence of functions all of which satisfy differential
equations of order at most two, then
$j$ is independent from $f_1, \ldots, f_n$ over $\CC$.   The case of $n = 0$ is trivial.  For the inductive
case of $n+1$, let $K := \CC \langle f_1, \ldots, f_n \rangle$ be the differential
field generated by $f_1, \ldots, f_n$ over $\CC$.  By induction, $j$ is free from $K$ over $\CC$ so that
$\tp(j/K)$ is minimal as well implying that if $j$ depends on $f_1, \ldots, f_{n+1}$ over $\CC$, then $j \in K \langle f_{n+1} \rangle^\text{alg}$,
but $\operatorname{tr.deg}_K(K \langle j \rangle) = 3 > 2 \geq \operatorname{tr.deg}_K (K \langle f_{n+1} \rangle^{\text{alg}})$.

Similar remarks apply to Pila's theorems~\cite{PilaDer}. Algebraic equations potentially satisfied by
$\{ j(g_i t), j'(g_it) , j'' (g_i t) \}_{i=1}^n$ where $g_i \in \operatorname{GL}_2 (\m C)$ are considered
over function fields which include exponential and Weierstrass $\wp$-functions. These satisfy first order and second order
 differential equations, so the remarks from the previous paragraph apply.
\end{rem}

The main theorem of Hrushovski's manuscript~\cite{HrushovskiJou} is that if the definable
set $X$ is defined by an order one differential equation over the constants and
is orthogonal to the constants, then the induced structure on $X$ over any finite
set of parameters over which it is defined is $\aleph_0$-categorical.  Under additional technical assumptions,
Rosen~\cite{RosenJou} proved the theorem without the hypothesis that $X$ is defined over
the constants (however, the technical assumptions are of a nature such that it is not obvious if they ever hold).   It has been known since the identification of Manin kernels that
not every strongly minimal which is orthogonal to the constants must have
$\aleph_0$-categorical induced structure, but the question of whether a strongly minimal
set with trivial forking geometry must have $\aleph_0$-categorical induced structure has remained open until now.

\begin{thm}\label{strongmin}
The set $X$ defined by the differential equation $E$ is
strongly minimal with trivial forking geometry but does not have
$\aleph_0$-categorical induced structure over any base set.
\end{thm}
\begin{proof}
Our main theorem, Theorem~\ref{stronglyminimal}, asserts that $X$ is strongly minimal.
Triviality of the forking geometry of the generic type of $X$ (and, hence, of
$X$ itself) is an immediate consequence of Pila's modular Ax-Lindemann-Weierstrass
theorem with derivatives.  Indeed, suppose that $a_1, \ldots, a_n \in X$ are $n$ pairwise independent
realizations of the generic type of $X$.  We shall check that they are independent
as a set.  By Theorem~\ref{RLEM} we may realize these points as
meromorphic functions of the form $j(g_i \cdot t)$ where $g_i \in \operatorname{GL}_2({\mathbb C})$.
By Pila's theorem, provided that $g_1, \ldots, g_n$ lie in distinct cosets of $\operatorname{GL}_2({\mathbb Q})$,
the differential field ${\mathbb C} \langle a_1, \ldots, a_n \rangle$ has transcendence degree
$3n$ over ${\mathbb C}$, which is the dimension of the field generated by a Morley sequence of
length $n$.  On the other hand, if $g_i = \delta g_j$ for some
$\delta \in \operatorname{GL}_2(\QQ)$ and $i \neq j$, then $a_i$ and $a_j$ are interalgebraic.  Indeed,
the image of the complex analytic variety $\{ (x,y) \in \fh^2 ~:~ y = \delta \cdot x \}$
under $(x,y) \mapsto (j(x),j(y))$ is a modular curve.
As we have assumed the $a_i$'s to be pairwise independent, we cannot have any geodesic relations
amongst the $g_i$'s.

On the other hand, the Hecke correspondences show that $X$ is not $\aleph_0$-categorical.
All of the functions $j(\gamma \cdot t)$ for $\gamma \in \operatorname{GL}_2({\mathbb Q})$ lie in $X$
and each is algebraic over $j$ as witnessed by the Hecke correspondence coming from $\gamma$.   As
$\operatorname{GL}_2({\mathbb Q})/\operatorname{GL}_2({\mathbb Z})$ is infinite, we
see that there are infinitely many elements of $X$ algebraic over the single element $j$.
Hence, $X$ does not have $\aleph_0$-categorical induced structure.
\end{proof}

\begin{rem} Suppose that $\Gamma \leq \SL_2( \m Z)$ is an arithmetic subgroup. One might inquire about the differential algebraic properties of $j_{\Gamma}$, where $j_\Gamma$ is the analytic function expressing $\Gamma \backslash \fh$ as an algebraic curve.  First, since $j_{\Gamma}$ is not algebraic and is interalgebraic with $j$ over $\m Q$, we can see that the type is strongly minimal. In differential algebra, generally, this would not be enough to conclude that the locus of the type is strongly minimal. However, Nishioka's theorem applies to automorphic functions in this setting, and so one sees that there are no order two or lower solutions to the differential equation satisfied by $j_{\Gamma}$. Further, for $g \in SL_2 ( \m C)$, we have the following diagram:

$$\xymatrix{  j_{\Gamma} (t) \ar@{-}[d]  \ar@{~}[r] & j_{\Gamma}(gt) \ar@{-}[d] \\
j(t) \ar@{.}[r] & j(gt)}$$
The solid vertical lines indicate interalgebraicity. The relationship of $j(t)$ and $j(gt)$ is completely controlled by the Pila's modular Ax-Lindemann-Weiestrass theorem \cite{PilaAO} (and by the stronger theorem, to the derivatives as well \cite{PilaDer}). It follows
from the interalgebraicity of $j(gt)$ and $j_\Gamma(gt)$, that if $g_1, \ldots, g_n$ lie in distinct cosets of $\SL_2(\QQ)$, then
the functions $j_{\Gamma} (g_1 t), \ldots, j_{\Gamma}(g_n t)$ are differentially algebraically independent.  That is,
the relationship (in terms of algebraic closure in the sense of differential fields) indicated by the top (curly) line is completely
controlled by the modular Ax-Lindemann-Weierstrass theorem and the results of this paper; naturally, one obtains as a by-product the Ax-Lindemann-Weierstrass theorem with derivatives for $j_{\Gamma} (t)$.

\end{rem}

\section{Other fibers}
\label{otherfibers}
In the previous sections, we investigated the properties of the the differential algebraic equation satisfied 
by the $j$-function, or in the language of~\cite{buiumIII}, we investigated the fiber of $\chi$ above $0$. 
In this section, we will investigate the other fibers as well as the possible algebraic relations across fibers 
in order to prove finiteness results.   The general problem will be reduced to an analytic one via Seidenberg's 
theorem combined with the special nature of the differential equations in question. 

\subsection{Minimality, strong minimality and other trivialities}
Fix $a_s$, an element in some differential field extension of $\QQ$.  
(Here the subscript ``$s$'' is meant to suggest the Schwarzian.  The reason
for this choice of notation should become clear shortly.)    By Seidenberg's Theorem~\ref{Seidenberg}
 we may realize the abstract differential field $\m Q \langle a_s \rangle $ as a differential 
 subfield of $\mc M (U)$ the field of meromorphic functions on some connected open subset $U$ of $\fh$. 
 We shall write the variable ranging over $U$ as $t$ and will write such expressions as $a_s (t)$ when
 we we wish to regard $a_s$ as a meromorphic function. Perhaps at the cost of shrinking the open domain $U$, 
 we may find some $\tilde a (t)$, an analytic function on some $U$ such that $\chi (j (\tilde a (t))) =  a_s (t)$, 
 as functions of $t$.    Alternatively, from the analytic
 description of $\chi$ we have that $a_s$ is the Schwarzian of $\tilde a$.  Define $a:= j ( \tilde a ) \in {\mathcal M}(U)$.

For a given derivation $\delta$, we remind the reader of our notation from the introduction:
$$\chi _\delta (x) := S_\delta (x) + \frac{x^2 - 1968 x +2\, 654\, 208}{2x^2 (x-1728)^2} (\delta x)^2 \text{ .}$$

The following obvious observation (which is an immediate consequence of the chain rule) 
will be used throughout the remainder of the section.

\begin{lem} \label{changeofvariables} If $V \subseteq U$ is a small enough connected open domain on which $\tilde a$ is 
 univalent and $K$ is a $\frac{d}{dt}$-differential subfield
of ${\mathcal M}(V)$ containing $a$ and $\tilde a$,  then $K$ is also a $\frac{d}{du}$-differential field where $u = \tilde a (t)$.
 Furthermore, $\chi_{\frac{d}{du}}(a)=0$.
\end{lem}

\begin{prop}
\label{othersm}
The set defined by the formula $\chi(x) = a_s$ is strongly minimal.  Moreover, if 
$a_1, a_2, \ldots a_n$ satisfy $\chi _\delta (a_i)  = a_s$ and $B$ is any algebraically closed differential field
containing $a_s$,  then $a_1 , a_2 , \ldots a_n$ is independent over $B$, unless there is some 
$k$ with $a_k \in B$ or there is a pair $i < j$  for which $F_N (a_1, a_2) = 0$ for some modular polynomial $F_N$.
\end{prop}

\begin{proof}
Let us first address strong minimality. 
It suffices to show that in some differentially closed field 
$\UU$ extending $\QQ \langle a_s \rangle$ the set 
$F_{a_s} := \{ x \in \UU ~:~ \chi_{\frac{d}{dt}}(x) = a_s \}$ is
 infinite but every every differentially constructible subset is finite or cofinite. 
Taking $\UU$ to be a differential closure of $\QQ \langle a_s, \tilde a \rangle$ and 
using Seidenberg's Theorem~\ref{Seidenberg} repeatedly, 
we may realize $\UU$ as a differential field of germs of meromorphic functions.
By the observation of Lemma~\ref{changeofvariables}, the differential field $\UU$ is 
also a differential field with respect to 
$\frac{d}{du}$ and the set $F_{a_s}$ is equal to $\{ x \in \UU ~:~ \chi_{\frac{d}{du}} (x) = 0 \}$.  
By the strong minimality of the equation for 
the $j$-function, this latter set is infinite and every $\frac{d}{du}$-differentially 
constructible subset is finite or cofinite.  In particular,
since every $\frac{d}{dt}$-differentially constructible set 
is a $\frac{d}{du}$-differentially constructible set, $F_{a_s}$ is 
strongly minimal.

For the ``moreover'' clause describing dependence amongst the solutions to $\chi(x) = a_s$, replacing 
$a_1, \ldots, a_n$ with realizations of the nonforking extension of 
$\tp(a_1,\ldots,a_n/B)$ to $B \langle \tilde a \rangle$, we may assume that $a_1, \ldots, a_n$
is independent from $\tilde{a}$ over $B$. 
Then as in the proof of strong minimality, we see that $a_1, \ldots, a_n$
satisfy $\chi_{\frac{d}{du}}(x) = 0$ and that any dependence must therefore come from 
algebraicity over $B \langle \tilde a \rangle$, which would reduce to algebraicity over the original $B$ by 
forking transitivity, or a Hecke relation.  
\end{proof}
 
\begin{rem} Proposition~\ref{othersm} characterizes algebraic closure in $\chi _\delta ^{-1} (a_s)$ and shows 
that the sets have trivial forking geometry which is not $\aleph_0$-categorical. 
\end{rem}

\subsection{Orthogonality}
\label{orthosect}
We begin this section with some standard notation from differential algebraic geometry, which we will require in the proof of our result.  The constructions of prolongation spaces and of the corresponding
differential sections are valid in a much more general context than what we present here where we specialize to
embedded affine varieties and work with coordinates.  Further details may be found in~\cite{MSJETS}.  Let $(K,\partial)$ be a differential field of characteristic zero and $n$ and $\ell$ a pair of natural
numbers.   We define the $\ell^\text{th}$ prolongation space of
affine $n$ space, $\tau_\ell {\mathbb A}^n$, to be the ${\mathbb A}^{n(\ell + 1)}$
where if the coordinates on ${\mathbb A}^n$ are $x_1, \ldots, x_n$, then the coordinates
on $\tau_\ell {\mathbb A}^n$ are $x_{i,j}$ for $1 \leq i \leq n$ and $0 \leq j \leq \ell$.
We define $\nabla_\ell:{\mathbb A}^n(K) \to \tau_\ell {\mathbb A}^n (K)$
by the rule
$$(a_1,\ldots,a_n) \mapsto (a_1,\ldots,a_n;a_1',\ldots,a_n'; \ldots; a_1^{(\ell)}, \ldots, a_n^{(\ell)})$$
where as above we write $x'$ for $\partial(x)$ and $x^{(\ell)}$ for $\partial^\ell (x)$.   If
$(L,\partial)$ is a differential field extension of $(K,\partial)$ we continue to write $\nabla_\ell$
for the corresponding map on ${\mathbb A}^n(L)$.

If $X \subseteq {\mathbb A}^n$ is an embedded affine variety and $S, T \subseteq \tau_\ell {\mathbb A}^{n}$
are two subvarieties of the prolongation space, then we define the differential constructible
set $(X,S \smallsetminus T)^\sharp$ by the rule
$$(X,S \smallsetminus T)^\sharp(K) := \{ a \in X(K) ~:~ \nabla(a) \in (S \smallsetminus T)(K) \} \text{ .}$$
 
The particular differential algebraic varieties which interest us are given by the fibers of $\chi$: 

$$\chi _\delta (x) := S_\delta (x) + \frac{x^2 - 1968 x +2\, 654\, 208}{2x^2 (x-1728)^2} (\delta x)^2 \text{ .}$$
Thus, $\chi^{-1} _\delta  ( a_s)$ is given by the set of $x$ such that 

$$(x''' x' - \frac{3}{2} (x'')^2)(2x^2 (x-1728)^2) + (x^2 - 1968 x +2\, 654\, 208) ( x')^4 = a_s (x')^2 (2x^2 (x-1728)^2)$$ 
and $x' \neq 0$ (note that this implies that $2x^2 (x-1728)^2 \neq 0$). In this case, $S$ is given 
by the above algebraic equation on $\tau _3 (\m A^1) = \m A^4$ and $T$ is given by t
he equation $x'=0$ in the same space.  We note that $S$ is an irreducible hypersurface of degree $6$.

When analyzing possible algebraic relations between collections of solutions 
(and their derivatives) to various fibers of $\chi$, the previous section gives 
a complete account of the algebraic relations within a given fiber. In this section, 
we prove that there are no algebraic relations across fibers:

\begin{thm}
\label{orthothm}
For $ b_s \neq c_s$, $\chi^{-1} ( b_s) \perp \chi^{-1} ( c_s)$.
\end{thm}

\begin{proof}
By Proposition~\ref{othersm}, each of the definable sets $\chi^{-1}(b_s)$ and $\chi^{-1}(c_s)$ is strongly minimal.  
Hence, $\chi^{-1}(b_s) \not \perp \chi^{-1}(c_s)$ is equivalent to the existence of a finite-to-finite correspondence between these fibers, possibly defined over new parameters. 

Let us fix some finitely generated differential ring $R$ over which all of these data are defined, and assume also that $R$ additionally contains $b$ and $c$ where $b = j( \tilde b)$  and $\tilde b$ is such that $b_s$ is the Schwarzian of $\tilde b$ (and similarly for $c$); that is, $R$ contains a solution to each of the two fibers of $\chi$. By quantifier elimination, the description of algebraic closure in differential fields, and the fact that the third derivative of a solution to a fiber of $\chi$ is rational over the previous derivatives, we may assume that the finite to finite correspondence $\Gamma_0$ is given by $\nabla _2 ^{-1} \Gamma$ where $\Gamma \subseteq \tau_2(\AA^1 \times \AA^1) = \AA^3 \times \AA^3$. By induction (it is only necessary to get a contradiction for irreducible correspondences) and possibly further localizing $R$, we may assume that $\Gamma$ is an absolutely integral $R$-scheme and that $\Gamma$ gives a finite-to-finite correspondence on $\m A ^3 _R \times \m A^3 _R$. 

Let $K = \text{Frac} (R)$. Applying Seidenberg's embedding theorem, there is some connected open $U \subseteq \fh$
 such that $K$ embeds into $\cM(U)$. Further shrinking $U$ if necessary, we may assume that $R \subseteq \mc O (U) $ and  
 that $\tilde b'$ has no zeroes on $U$.
 For the remainder of the proof, we replace $R$ with $\mc O (U)$ so that
 we may identify $\operatorname{Spec}(R)$ with $U$.  Take $t_0 \in U$ and consider
 the fiber$\Gamma _{t_0}(\m C) \subseteq \tau _2 ( \m A^ 1 \times \m A^1 ) (\m C)$.

For a  review of the prolongation spaces and their relation to differential geometric jet spaces,
see sections 2.1 and 2.2 of~\cite{scanlon2014covering}. Taking differential geometric jets we obtain a map $J_2(j):J_2(\fh) \to \tau_2(\AA^1) (\m C)$  which 
fits into the following commutative diagram.

$$
\xymatrix{
\fh \times \fh  \ar[d]^{j \times j} &  & J_2(\fh \times \fh) \ar[d]^{J_2(j \times j)} \ar[ll]^{\pi} & J_2(\fh) \times J_2(\fh) \ar@{=}[l]  \ar[d]^{J_2(j) \times J_2(j)}\\
(\AA^1 \times \AA^1) (\m C) & &\tau_2(\AA^1 \times \AA^1) ( \m C )  \ar[ll]^{\pi} & (\AA^3 \times \AA^3) ( \m C) \ar@{=}[l]  
}
$$ 

\begin{cl}\label{densefiber} The set 
$$A := \{ (x ( t_0) , x ' (t_0 ), x ''(t_0), y ( t_0 ), y' (t_0), y''(t_0)) \in \Gamma (\m C)  
\, | \, \chi (x) = b_s , \, \chi (y) = c_s \}$$ is Zariski dense in $\Gamma_{t_0}$. 

\end{cl} 
\begin{pfcl} The projection of $F$ to $\m C^3$ contains the set

$$B :=\{ (j(g \cdot \tilde{b}(t_0)), \frac{d}{dt} (j(g \cdot \tilde{b}(t)) \vert_{t=t_0}, 
\frac{d^2}{dt^2} (j(g \cdot \tilde{b}(t)) \vert_{t = t_0}) ~\vert~ g \in \GL_2^+(\RR) \}$$
which is Zariski dense in $\m C^3$.  Indeed, because $\GL_2^+(\RR)$ acts transitively on 
$\fh$, we see that the projection of $B$ to the first coordinate is all of $\CC$.  The fibers in the 
tangent space over any such point are obtained by restricting $g$ to the stabilizer of some point in $\fh$.  
From the formula for the derivative, it is clear that the image of $B$ is dense in these fibers as well.  Likewise, 
the stabilizer of any point in $J_1(\fh)$ is one dimensional and again the formula for the second derivative 
shows that image of $A$ in the fibers of $J_2(\fh)$ over $J_1(\fh)$ is dense.  

Since 
$\Gamma_{t_0}$ is an irreducible finite-to-finite correspondence, the set $F$ 
is Zariski dense in $\Gamma_{t_0}$.  
\end{pfcl}

Note that the action of $\GL_2^+(\m R) \times \GL_2^+( \m R)$ on $\fh_ \times \fh$ extends 
canonically to an action on $J_2(\fh \times \fh)$ via the jets.   
Let $\widetilde{(\Gamma_{t_0})}$ be a component of $(J_2(j) \times J_2(j))^{-1} \Gamma_{t_0}$ and let 
$H_{t_0}$ be the setwise stabilizer of $\widetilde{ (\Gamma_{t_0})  }$ in $\GL_2^+( \m R) \times \GL_2^+( \m R)$.

\begin{cl}
\label{rationalstab}
For each $\gamma \in \GL_2^+(\QQ)$ there is some $\delta \in \GL_2^+(\QQ)$ with 
$(\gamma,\delta) \in H_{t_0}$  and likewise for each $\delta \in \GL_2^+(\QQ)$ there is 
some $\gamma \in \GL_2^+(\QQ)$ with $(\gamma,\delta) \in H_{t_0}$.
\end{cl}

\begin{pfcl}
Let $\gamma \in \GL_2^+(\QQ)$.  We know that the image under $j \times j$ of the graph of the action of 
$\gamma$ on $\fh$ is an 
algebraic correspondence on $\AA^1 \times \AA^1$ which restricts to a finite-to-finite correspondence from 
$\chi^{-1}(b_s)$ to itself.   The image of this correspondence under $\nabla_2^{-1} (\Gamma_{t_0})$ is thus 
a finite-to-finite correspondence from $\chi^{-1}(c_s)$ to itself.  By Proposition~\ref{othersm}, 
this new correspondence must be given by a finite union of Hecke relations which are themselves
images under $j \times j$ of graphs of the action of some $\delta_1, \ldots, \delta_n \in \GL_2^+(\QQ)$.

By Claim~\ref{densefiber}, there is a Zariski dense set of points $(x,y)$ in $\Gamma_{t_0}(\m C)$ so that 
for any $u$ with $(x,u)$ in the Hecke correspondence coming from $\gamma$, there is some 
$v$ with $(y,v)$ in the Hecke correspondence coming from from $\delta_i$ for some $i \leq n$ and
$(u,v) \in \Gamma_{t_0}(\m C)$.  As this is an algebraic condition, it holds everywhere on $\Gamma_{t_0}$.
Thus, for any $(x,y) \in \widetilde{(\Gamma_{t_0})}$, there is some $i \leq n$ and some $\epsilon \in \SL_2(\ZZ)$ 
such that $(\gamma \cdot x, \epsilon \delta_i \cdot y) \in \widetilde{(\Gamma_{t_0})}$.  For any given 
$\delta \in \GL_2^+(\QQ)$ the set $\widetilde{(\Gamma_{t_0})} \cap (\gamma^{-1},\delta^{-1}) \cdot 
\widetilde{(\Gamma_{t_0})}$ is a closed analytic subset of $\widetilde{(\Gamma_{t_0})}$.  As $\widetilde{(\Gamma_{t_0})}$
is irreducible and may expressed as the countable union of such intersections  we have 
$\widetilde{(\Gamma_{t_0})} = \widetilde{(\Gamma_{t_0})} \cap (\gamma^{-1},\delta^{-1}) \cdot \widetilde{(\Gamma_{t_0})}$ for 
some $\delta \in \GL_2^+(\QQ)$.  That is, 
$(\gamma,\delta) \in H_{t_0}$.  Arguing with the 
first and second coordinates reversed we obtain the ``likewise'' clause.
\end{pfcl}

Let us write $\overline{H_{t_0}}$ for the image of $H_{t_0}$ in $\PSL_2( \m R) \times \PSL_2( \m R)$.  Note that 
$\overline{H_{t_0}}$ is the setwise stabilizer of $\widetilde{(\Gamma_{t_0})}$ in $\PSL_2( \m R) \times \PSL_2( \m R)$.
From Claim~\ref{rationalstab} and the fact that the image of $\GL_2^+(\QQ)$ is dense in $\PSL_2 (\m R)$ we see that the 
projection of $\overline{H_{t_0}}$ to each $\PSL_2( \m R)$ is surjective.  Since $\widetilde{\Gamma_{t_0}}$
is a finite-to-finite correspondence between $J_2(\fh)$ and $J_2(\fh)$, necessarily $\overline{H_{t_0}}$ 
is a proper subgroup of $\PSL_2( \m R) \times \PSL_2( \m R)$.  Arguing as in~\cite{Kolchinagad} 
we see that $\overline{H_{t_0}}$ is the graph of an automorphism of $\PSL_2( \m R)$.  Since every 
automorphism of $\PSL_2( \m R)$ is inner, we can find some $g \in \PSL_2( \m R)$ for which 
$\overline{H_{t_0}} = \{ (\gamma, \gamma^g) : \gamma \in \PSL_2(\m R) \}$.

Let us consider some point $(x,y) \in \widetilde{\Gamma _{t_0} }$.  Write $\pi(x,y) =: (x_0,y_0) 
\in \pi(\widetilde{\Gamma _{t_0}}) \subseteq (\fh \times \fh)$.  Let $k \in \PSL_2(\m R)$ with 
$k \cdot x_0 = y_0$.  We will show that we may take $k = g$.    

Let us write the stabilizer of $x$ in $\PSL_2(\m R)$ as $S_x$.   Note that if $h \in S_x$, 
then because $(h,h^g) \in H$, we have  $(x, h^g \cdot y) = (h,h^g) \cdot (x,y) \in \widetilde{\Gamma_{t_0}} $.
Since $\widetilde{(\Gamma_{t_0})}$ is a
finite-to-finite correspondence, the fiber of $\widetilde{(\Gamma_{t_0})}$ above $x$ is finite.  Hence, 
the orbit $S_x^g \cdot y$ is finite.  That is, the group $S_x^g \cap S_y$ has finite 
index in $S_x^g$, but this last group is a connected group so it follows  that $S_x^g \leq S_y$.  
Projecting $\pi: J_2 (\fh ) \to \fh$ we conclude that $S_{x_0}^g = \pi(S_x^g) \leq  \pi(S_y) = 
S_{y_0} = S_{x_0}^k$.   Since the group $S_x$ is self-normalizing, we conclude that 
$g S_{x_0} = k S_{x_0}$.  That is, it would have been possible to take $k = g$.
Thus, $\pi(\widetilde{(\Gamma_{t_0})})$ is the graph of the action of $g$ on $\fh$.

Since $J_2(j \times j) (\widetilde{(\Gamma_{t_0})}) = \Gamma_{t_0}$ is an algebraic variety, necessarily 
$g \in \GL_2^+(\QQ)$.  As there are only countably many Hecke relations, it follows that 
one must hold for the generic fiber of $\Gamma$. 
This finishes the proof of the theorem, because it is a contradiction to $b_s \neq c_s$. 
\end{proof}

\section{Effective finiteness results}

In this section, we apply our earlier work on the strong minimality of the
differential equation satisfied by the $j$-function to compute explicit upper
bounds on certain intersections of isogeny classes of products of elliptic curves
with algebraic varieties.   As we explained in the introduction to this paper,
the questions we address were posed to us by Mazur in connection with theorems of Orr
in line with the Zilber-Pink conjectures.   In~\cite{orr2012families}, Orr proves the
following theorem.

\begin{thm}[Orr, Theorem 1.3] \label{thmorr}
Let $\Lambda$ be the isogeny class of a point $s \in {\mathcal A}_g(\CC)$, the
moduli space of principally polarized abelian varieties of dimension $g$.  Let
$Z$ be an irreducible closed subvariety of ${\mathcal A}_g$ such that $Z \cap \Lambda$
is Zariski dense in $Z$ and $\dim(Z) > 0$.

Then there is a special subvariety $S \subseteq {\mathcal A}_g$ which is isomorphic to
a product of Shimura varieties $S_1 \times S_2$ with $\dim S_1 > 0$ and such that
$$Z = S_1 \times Z' \subseteq S$$
for some irreducible closed $Z' \subseteq S_2$.
\end{thm}

In the case that $Z$ is a curve, Theorem~\ref{thmorr} implies that $Z$ must be a
weakly special variety.  We refer the reader to the original paper for a discussion of
special and weakly special varieties, but note that if $S \subseteq {\mathcal A}_g$ is the
subvariety corresponding to the abelian varieties expressible as a product of $g$ elliptic
curves, then identifying $S$ with ${\mathbb A}^n$,
the special subvarieties of $S$ are the components of varieties defined by equations of the
form $F_N(x_i,x_j) = 0$ where $F_N$ is a modular polynomial and $1 \leq i \leq j \leq n$.
The weakly special varieties are obtained by allowing equations of the form $x_k = \zeta$ for
some $\zeta \in {\mathbb A}^1(\CC)$.

Taking the contrapositive of Theorem~\ref{thmorr}, again for curves, one sees that if
$Z \subseteq {\mathcal A}_g$ is an algebraic curve which is not weakly special, then
$Z \cap \Lambda$ is finite.   One might ask how large is this finite set.  Since
Orr's argument depends on ineffective constants coming from the Pila-Wilkie o-minimal
counting theorem, it does not yield a method to compute a bound on $Z \cap \Lambda$.
Using differential algebraic methods, we can find explicit upper bounds depending only on
geometric data, but we must restrict our attention to transcendental points. 

In this section, we will begin with a specific example of an application of the work in the previous sections, 
before giving a general theorem. The general idea we are following is a familiar one in the model theory of fields 
(e.g. \cite{HrML, buiumIII}). Take a set $A$ which has some arithmetic meaning 
(in our case, isogeny classes viewed in a moduli space); we wish to study the intersection of $A$ with varieties. 
Instead of considering the intersections directly, take the closure $\overline{A}^{\operatorname{Kol}}$ 
of  $A$ in the Kolchin topology, and study intersections of $\overline{A}^{\operatorname{Kol}}$ with varieties. 
The sacrifice which one makes in moving to the Kolchin closure is offset by a reasonable understanding of the 
properties of the closure, which we accomplished in the previous sections. The advantage is that the object 
in question is now a variety in the sense of differential algebraic geometry, so we can apply tools and 
uniformities from the general theory.

The sort of problem which we are attacking has, on the face of it, nothing to do with differential algebra. 
This allows us a good deal of freedom in equipping the fields over which we are working with a derivation. 
Equip $\m C$ with a derivation $\partial$ so that that $(\m C, \partial)$ is differentially
closed and the field of constants of $(\m C, \partial)$ is $\m Q^{\operatorname{alg}}$. 
Given a particular isogeny class viewed in the moduli space of elliptic curves, in order to 
apply the results of the previous sections, we must know that the elements in the class satisfy 
the differential equation $\chi (x) = a$ for some $a$ in the differential field. 
This is possible precisely when the element is transcendental. 

For background on the theory of moduli spaces of elliptic curves, we refer to~\cite{milneMF}. 

One key tool is an effective finiteness theorem of Hrushovski and Pillay from~\cite{HPnfcp} and Le\'on-Sanchez and Freitag~\cite{JOnfcp}.

\begin{thm}\label{HPnfcp} 
Let $X$ be a closed subvariety of ${\mathbb A}^n$, with $\dim(X)=m$, and let $S, T$ be closed subvarieties 
(not necessarily irreducible) of
$\tau_\ell {\mathbb A}^n$.   Then the degree of the Zariski closure of $(X,S \smallsetminus T)^\sharp(\m C, \partial)$
is at most $\deg(X)^{\ell 2^{m \ell}} \deg(S)^{2^{m \ell} - 1}$.   In particular, if 
$(X,S \smallsetminus T)^\sharp(\m C, \partial)$ is a finite set, this expression bounds the number of 
points in that set. 
\end{thm}

\begin{rem} A noneffective proof of the non-finite cover property in differential fields, that is, the
assertion that for a differentially constructible family of differentially constructible set $\{ X_b \}_{b \in B}$
over a differentially closed field ${\mathbb U}$ of characteristic zero there is a bound $N$ so that for any $b$ is $X_b({\mathbb U})$ is finite,
then $|X_b({\mathbb U}))| \leq N$,  is in \cite[Dave Marker's differential fields article]{MMP}.
\end{rem}

\begin{rem}
There are some gaps in the proof of the theorem in~\cite{HPnfcp}.  In particular,
there is a false (implicit) claim, which results in some additional complications 
in the proof.  We follow the notation and numbering of that paper during this remark.  One of the errors comes
from their assumption \cite[Lemma 3.6 (2)]{HPnfcp} that $B_\ell (X)$ is a component of the $\ell^{\text{th}}$
 prolongation space of the variety 
$X$; following their proof, the justification is that $B_\ell(X)$ is a subset of the $\ell^{\text{th}}$
prolongation and $B_\ell(X)$ has 
dimension $\dim(X) (\ell+1)$. Of course, this justification would only be valid if one assumed that the $\ell^{\text{th}}$ prolongation had 
dimension $\dim(X) (\ell+1)$. 

However, in general $B_\ell(X)$ can have
smaller dimension than the $\ell^{\text{th}}$ prolongation 
of $X$; when $X$ is defined over the constants, the assumption is equivalent to assuming
that the $\ell^{\text{th}}$ prolongation is equidimensional, which is not true in general. The dimension of the 
$\ell^{\text{th}}$ prolongation is related to
some basic invariants of the singularities on $X$, in this case, the log canonical threshold~\cite{logcanon}. 

This small assumption and another small gap in the proof results in additional complications.  These gaps are corrected and the problem is generalized to the partial differential context in~\cite{JOnfcp} 
using the theory of prolongation spaces~\cite{MSJETS}. 
\end{rem}

\begin{convention}
When we compute degrees of closed subsets of affine space 
in what follows, we take the \emph{definition} of degree to be the sum of the geometric degrees of the irreducible components. 
\end{convention}

Our second ingredient is the fact that the fibers of $\chi$ are invariant 
under isogeny.  In fact, by Theorem~\ref{othersm}, we know that for any nonconstant $a$, 
the equation $\chi (x) = \chi(a)$ holds on the isogeny class of $a$. (In fact, by a theorem of Buium~\cite{buiumIII}, it 
defines the Kolchin closure of the isogeny class of $a$.) The third ingredient is our characterization algebraic relations across the fibers of $\chi$ from the previous Sections~\ref{orthothm}.  We begin with a specific example.

\subsection{Automorphisms of the Riemann sphere}
The ingredients and general strategy are essentially the same. 
Fix some non-identity element of $GL_2(\CC)$, which we will write as $\alpha = \left( \begin{array}{cc} a & b \\ c & d \end{array} \right)$. 
Throughout this section, we let $E_x$ denote an elliptic curve with $j$-invariant $x$ and we write 
$x \sim y$ to mean that the elliptic curves $E_x$ and $E_y$ are isogenous.

Fix $\tau$ transcendental. The goal here is to establish an upper bound on the number of elliptic curves $E_\eta$
such that $E_\tau \sim E_\eta$ and $E_{\alpha \cdot \tau } \sim E_ {\alpha \cdot \eta }$. 

Unless $\alpha \cdot \tau$ is constant (a case we will consider separately),
since the fibres of $\chi$ are closed under isogeny, this set is contained in the following set:

$$S=\{ z \in {\m A}^1(\CC)  ~\vert~ \chi(z) = \chi(\tau) \text{ and } \chi ( \alpha \cdot z ) ) = \chi (\alpha  \cdot \tau ) \} \text{ }$$

which is the projection to the first coordinate of the intersection of the graph of $\alpha$ (regarded 
as  linear fractional transformation) with $\chi^{-1} (\chi(\tau)) \times \chi^{-1} (\chi(\alpha \cdot \tau))$.  
From Theorem~\ref{orthothm}, because the graph $\alpha$ is not a modular relation, we know that this intersection is finite. 
In the case that $\alpha \cdot \tau$ is constant, we observe that the isogeny class of $\alpha \cdot \tau$ 
is contained in the set of constants, which is itself a stronlgy minimal set orthogonal to $\chi^{-1} (\chi(\tau))$ because
the former is nontrivial while the latter is trivial.  Hence, for the same reason this set is finite. 

Next, we apply Theorem~\ref{HPnfcp}.  We explain the details in the case where $\alpha \cdot \tau$ is 
transcendental. The other case is even easier.

Let $X = \m A^1$ and let $\ell=3$, and write $B_3 ( \m A^1 ) $ in coordinates $(z, \dot z , \ddot z, \dddot z)$.
Let $S$ be given the equations $\chi(z) = \chi(\tau)$ and $\chi (\alpha \cdot z )  = \chi ( \alpha \cdot \tau)$ 
re-expressed as algebraic equations in $z, \dot z, \ddot z, \dddot z$.
By Bezout's theorem, $S$ is of degree at most $36$. By Theorem~\ref{HPnfcp}, 
$|(X, S )^ \sharp| \leq 36^7$. Hence, given an elliptic curve $E_\tau$ with transcendental $j$-invariant, 
there are at most $36^7$ elliptic curves $E_\eta$ in the isogeny class of $E_\tau$ for which 
$E_{\alpha \cdot \eta }$ is in the isogeny class of $E_ {\alpha \cdot \tau}$.

\begin{rem} We will state our results in full generality, but we should point out that in certain special cases, better bounds are available via comparably elementary reasoning.
\end{rem}

\begin{prop} Let $C \subseteq {\mathbb A}^2$ be some non-weakly-special irreducible curve defined over ${\mathbb Q}^{\text{alg}}$ 
and let $P = (a,b)$ be some transcendental point. Then there can be at most one point in the isogeny class of $P$ on $C$. 
\end{prop} 
\begin{proof}
Without loss of generality, we may assume that $P \in C$.   
Suppose that $(a',b') \in C$ is distinct from $(a,b)$ but isogenous to $(a,b)$ 
via isogenies of degrees $n$ and $m$, respectively. 
Then $P$ would belong to the intersection of $C$ with the transform of $C$ by the correspondence 
$\{ (x,y), (u,v) : F_n(x,u) = 0 ~\&~ F_m(y,v) = 0 \}$.   As $C$ is non-weakly special, this intersection is 
is zero dimensional and defined over ${\m Q}^\text{alg}$, contradicting the presence of $P$ on the intersection. 
\end{proof}

\subsection{A general finiteness result}
The following general result uses similar tools to the examples of the previous two subsections. 
For the remainder of the paper, we will be considering Kolchin closed subvarieties 
$V$ of $\prod_{i=1}^n \chi ^{-1} (a_i)$, so we may assume, without loss of generality, 
that $V$ is written as $\nabla_3^{-1} S \cap \prod_{i=1}^n \chi ^{-1} (a_i)$
for an an algebraic subvariety $S$ of $\tau_3 (\m A^n)$. 
For the purposes of stating the theorem, we will \emph{define} $\deg (V) := \deg(S)$. 

\begin{thm} Let $V \subseteq \m A^n$ be a Kolchin closed subset. Let $\bar a$ be an $n$-tuple 
of transcendental points. Let 
$$\operatorname{Iso}(\bar a) := \{ (b_1,\ldots,b_n) \in {\mathbb C}^n ~\vert~ a_i \sim b_i \text{ for } i \leq n\}$$
be the isogeny class of $\bar a$.  Let $W$ be the Zariski closure of $V \cap \operatorname{Iso}(\bar a)$.

Then 
\begin{enumerate} 
\item $W$ is a finite union of weakly special subvarieties of $\m A^n$. 
\item The degree of $W$ is bounded by 
$\left( 6^n \cdot \deg (V) \right)^7 \text{ .}$
\item $V \cap \text{Iso} (\bar a) = W \cap \text{Iso} (\bar a)$. 
\end{enumerate} 
\end{thm} 

\begin{proof} 
Since the fibers of $\chi$ are closed under  isogeny,  $V \cap \operatorname{Iso}(\bar a)$
is contained in $V \cap \prod _{i=1}^ n \chi ^{-1} (a_i)$. 
By our orthogonality Theorem~\ref{orthothm} and our description of dependence within fibers from Theorem~\ref{othersm},
$V \cap \prod _{i=1}^ n \chi ^{-1} (a_i)$ is equal to $\bigcup_{j=1}^m X_j \cap \prod_{i=1}^n \chi^{-1}(a_i)$ 
where each $X_i$ is an irreducible weakly special variety.  
It is easy to see that if an irreducible weakly special variety $X$ meets 
$\operatorname{Iso}(\bar a)$ non-trivially, then $X \cap \operatorname{Iso}(\bar a)$ is Zariski dense in $X$.  Hence,
$W$, the Zariski closue of $V \cap \operatorname{Iso}(\bar a)$, is equal to $\bigcup_{j \in J} X_j$ for some $J \subseteq 
\{ 1, \ldots, m \}$.    
 
Write $V$ as $(\m A^n , \Xi_{\bar a} \cap S)^ \sharp$ where $\Xi_{\bar a}$ 
is given by the equations $\chi(x_j) = \chi(a_j)$ for $j \leq n$ in which $\chi(x_j)$ is re-expressed as 
a rational function in $x_j$, ${\dot x}_j$, ${\ddot x}_j$ and ${\dddot x}_j$.  
Examining the explicit equations for $\chi$, one sees that $\deg(\Xi_{\bar a}) = 6^n$.  
By computing degrees and applying Theorem~\ref{HPnfcp}, the degree of the Zariski closure of $V$ is bounded by 
$\left( 6^n \cdot \deg (S) \right) ^7$.  As $W$ is a union of some of the components of this Zariski closure, this number 
also bounds $\deg(W)$.
\end{proof} 

When $V$ is actually an algebraic variety, then we have $S = \tau_3 V$ so that $\deg(V)$ as defined 
with $\deg(S)$ is the same as $\deg(V)$ as usually defined. Thus, we obtain:

\begin{corr}  Let $V \subseteq \m A^n$ be a Zariski closed subset. Let $\bar a$ be an $n$-tuple of transcendental points.
Let $W$ denote the Zariski closure of $V \cap \operatorname{Iso} (\bar a)$.
Then $W$ is a finite union of weakly special subvarieties. The degree of $W$ is at most $\left( 6^n \cdot \deg (V) \right) ^7$.
\end{corr}

\bibliography{Research1}{}

\begin{thebibliography}{10}

\bibitem{buiumIII}
Alexandru Buium.
\newblock Geometry of differential polynomial functions, {III}: moduli spaces.
\newblock {\em American Journal of Mathematics}, pages 1--73, 1995.

\bibitem{Casanovas}
Enrique Casanovas.
\newblock {\em Simple theories and hyperimaginaries}.
\newblock Number~39. Cambridge University Press, 2011.

\bibitem{ChHr-AD1}
Zo{\'e} Chatzidakis and Ehud Hrushovski.
\newblock Difference fields and descent in algebraic dynamics, {I}, {I}{I}.
\newblock {\em Journal de l'Institut de Mathematiques de Jussieu}, 7:653--686,
  2008.

\bibitem{JOnfcp}
James Freitag and Omar~Le\'on Sanchez.
\newblock Effective bounds in partial differential fields.
\newblock In preparation, 2014.

\bibitem{Gaussisold}
Carl~Friedrich Gau{\ss}.
\newblock Disquisitiones arithmeticae, 1801. english translation by {A}rthur
  {A}. {C}larke, 1986.

\bibitem{HrushovskiJou}
Ehud Hrushovski.
\newblock {O}{D}{E}s of order 1 and a generalization of a theorem of
  {J}ouanalou.
\newblock Unpublished manuscript, 1995.

\bibitem{HrML}
Ehud Hrushovski.
\newblock The {M}ordell-{L}ang conjecture for function fields.
\newblock {\em Journal of the American Mathematical Society}, 9:no. 3,
  452--464, 1996.

\bibitem{HrIt}
Ehud Hrushovski and Masanori Itai.
\newblock On model complete differential fields.
\newblock {\em Transactions of the American Mathematical Society}, Volume
  355(11):4267--4296, 2003.

\bibitem{HPnfcp}
Ehud Hrushovski and Anand Pillay.
\newblock Effective bounds for the number of transcendental points on
  subvarieties of semi-abelian varieties.
\newblock {\em American Journal of Mathematics}, 122(3):439--450, 2000.

\bibitem{MneqU}
Ehud Hrushovski and Thomas Scanlon.
\newblock Lascar and {M}orley ranks differ in differentially closed fields.
\newblock {\em Journal of Symbolic Logic}, 64:no. 3,1280--1284, 1986.

\bibitem{HrSo}
Ehud Hrushovski and \v{Z}eljko Sokolovi\'{c}.
\newblock Strongly minimal sets in differentially closed fields.
\newblock unpublished manuscript, 1993.

\bibitem{Kolchinagad}
Ellis~R Kolchin.
\newblock Algebraic groups and algebraic dependence.
\newblock {\em American Journal of Mathematics}, pages 1151--1164, 1968.

\bibitem{Lascar}
Daniel Lascar.
\newblock Ranks and definability in superstable theories.
\newblock {\em Israel Journal of Mathematics}, 23(1):53--87, 1976.

\bibitem{mahler1969algebraic}
K~Mahler.
\newblock On algebraic differential equations satisfied by automorphic
  functions.
\newblock {\em Journal of the Australian Mathematical Society},
  10(3-4):445--450, 1969.

\bibitem{DOP}
David Marker.
\newblock The number of countable differentially closed fields.
\newblock {\em Notre Dame Journal of Formal Logic}, 48(1):99--113, 2007.

\bibitem{MMP}
David Marker, Margit Messmer, and Anand Pillay.
\newblock {\em Model theory of fields}, volume~5 of {\em Lecture Notes in
  Logic}.
\newblock Association for Symbolic Logic, La Jolla, CA; A K Peters, Ltd.,
  Wellesley, MA, second edition, 2006.

\bibitem{masserheights}
David Masser.
\newblock Heights, transcendence, and linear independence on commutative group
  varieties.
\newblock In {\em Diophantine approximation ({C}etraro, 2000)}, volume 1819 of
  {\em Lecture Notes in Math.}, pages 1--51. Springer, Berlin, 2003.

\bibitem{MilneEC}
James~S. Milne.
\newblock {\em Elliptic curves}.
\newblock BookSurge Publ., 2006.

\bibitem{milneMF}
James~S. Milne.
\newblock Modular functions and modular forms (v1.30), 2012.
\newblock Available at \url{www.jmilne.org/math/}.

\bibitem{MSJETS}
Rahim Moosa and Thomas Scanlon.
\newblock Jet and prolongation spaces.
\newblock {\em Journal de l'Institut de Mathematiques de Jussieu},
  9(2):391--430, 2010.

\bibitem{logcanon}
Mircea Musta{\c{t}}{\u{a}}.
\newblock Singularities of pairs via jet schemes.
\newblock {\em Journal of the American Mathematical Society}, 15(3):599--615,
  2002.

\bibitem{Nishioka}
Keiji Nishioka.
\newblock A conjecture of {M}ahler on automorphic functions.
\newblock {\em Archiv der Mathematik}, 53(1):46--51, 1989.

\bibitem{orr2012families}
Martin Orr.
\newblock Families of abelian varieties with many isogenous fibres.
\newblock {\em Journal f{\"u}r die reine und angewandte Mathematik (Crelles
  Journal)}, 2013.
\newblock \url{DOI 10.1515/ crelle-2013-0058}.

\bibitem{PilaAO}
Jonathan Pila.
\newblock O-minimality and the {A}ndr{\'e}-{O}ort conjecture for {$\mathbf
  C^n$}.
\newblock {\em Ann. of Math.(2)}, 173(3):1779--1840, 2011.

\bibitem{PilaDer}
Jonathan Pila.
\newblock Modular {A}x-{L}indemann-{W}eierstrass with derivatives.
\newblock {\em Notre Dame J. Formal Logic}, 54 (Ol\'{e}ron Volume):553--565,
  2013.

\bibitem{GST}
Anand Pillay.
\newblock {\em Geometric Stability Theory}.
\newblock Oxford University Press, 1996.

\bibitem{RosenJou}
Eric Rosen.
\newblock Order 1 strongly minimal sets in differentially closed fields.
\newblock \url{arXiv:math/0510233}, 2007.

\bibitem{notmin}
Maxwell Rosenlicht.
\newblock The nonminimality of the differential closure.
\newblock {\em Pacific J. Math.}, 52:529 -- 537, 1974.

\bibitem{Sacks}
Gerald Sacks.
\newblock {\em Saturated Model Theory}.
\newblock Mathematics Lecture Note Series. W. A. Benjamin, Inc., Reading,
  Massachussetts, 1972.

\bibitem{scanlon2014covering}
Thomas Scanlon.
\newblock Algebraic differential equations from covering maps.
\newblock \url{arXiv:1408.5177}, 2014.

\bibitem{seidenberg}
A.~Seidenberg.
\newblock Abstract differential algebra and the analytic case.
\newblock {\em Proc. Amer. Math. Soc.}, 9:159--164, 1958.

\end{thebibliography}
\bibliographystyle{plain}

\end{document}